\documentclass[12pt]{article}
\usepackage{latexsym}
\usepackage{amsmath}
\usepackage{amsfonts}
\usepackage{amsthm}
\usepackage{amssymb}

\def\0{\mathbf 0}
\def\1{\mathbf 1}

\numberwithin{equation}{section}
\newcounter{thm}[section]
\numberwithin{thm}{section}

\newtheorem{theorem}[thm]{Theorem}
\newtheorem{lemma}[thm]{Lemma}

\newtheorem{corollary}[thm]{Corollary}

\newtheorem{definition}[thm]{Definition}

 \title{Measure Algebras}

 \author{Thomas Jech \\
    e-mail: jech@math.psu.edu
    }

\begin{document}

\maketitle

These notes for the 2016 Winter School in Hejnice contain the complete proof of the following result \cite{J}:

\begin{theorem} \label{Thm}
A Boolean $\sigma$-algebra $B$ is a measure algebra if and only if
it is weakly distributive and uniformly concentrated.
\end{theorem}

\section{Boolean $\sigma$-algebras}\label{S1}

A \emph{Boolean algebra} is an algebra $B$ of subsets of a given
nonempty set $S$, with Boolean operations $a \cup b$, $a \cap b$,
$-a=S-a$, and the zero and unit elements $\mathbf 0=\emptyset$ and
$\mathbf 1=S$. A \emph{Boolean $\sigma$-algebra} is a Boolean
algebra $B$ such that every countable set $A\subset B$ has a
supremum $\sup A = \bigvee A$ (and an infimum $\inf A = \bigwedge
A$) in the partial ordering of $B$ by inclusion.

\begin{definition}\label{meas}
A \emph{measure} (a
 strictly positive
$\sigma$-additive probability measure) on a Boolean
$\sigma$-algebra $B$ is a real valued function $m$ on $B$ such
that
\begin{itemize}
\item [(i)] $m(\mathbf 0) = 0,\ m(a) > 0 \text{ for } a \not = \mathbf 0
        \text{, and } m(\mathbf 1) = 1 $
\item [(ii)]$ m(a)\le m(b) \text{ if } a\subset b$
\item [(iii)] $m(a \cup b) =m(a) + m(b) \text{ if } a\cap b =\mathbf
0$
\item [(iv)]
$m(\bigvee_{n=1}^\infty a_n) = \sum_{n=1}^\infty m(a_n)
        \text{ whenever the } a_n \text{ are pairwise disjoint.}$
\end{itemize}

A \emph{measure algebra} is a Boolean $\sigma$-algebra that
carries a measure.
\end{definition}

A function $m$ that satisfies (i), (ii) and (iii) above is called a
\emph{(strictly positive) finitely additive  measure}. And $m$
satisfies (iv) if and only if it is \emph{continuous}:
\begin{itemize}
\item [(iv)] if  $\{a_n\}_n$ is a decreasing sequence in $B$ with
$\bigwedge_{n=1}^\infty a_n = \0$ then $\lim_n m(a_n)=\0.$
\end{itemize}

Let $B$ be a Boolean algebra and let $B^+ = B - \{\mathbf 0\}$. A
set $A \subset B^+$ is an \emph{antichain} if $a\cap b =\mathbf 0$
whenever $a$ and $b$ are distinct elements of $A$. A
\emph{partition} $W$ (of $\mathbf 1$) is a maximal antichain, i.e.
an antichain with $\bigvee W =\mathbf 1$. $B$ satisfies the
\emph{countable chain condition} (ccc) if it has no uncountable
antichains. Every measure algebra is ccc because $B^+=
\bigcup_{n=1}^\infty \{C_n\}_n$ where $C_n=\{a:m(a)\ge 1/n\}$
and every antichain in $C_n$ has size at most $n$.

If $\{a_n\}_n$ is a sequence in a Boolean $\sigma$-algebra $B$,
one defines $$ \limsup_n a_n = \bigwedge_{n=1}^\infty \bigvee_{k =
n}^\infty a_k \text{ and }\liminf_n a_n = \bigvee_{n=1}^\infty
\bigwedge_{k = n}^\infty a_k,$$ and if $\limsup_n a_n = \liminf_n
a_n =a$, then $a$ is the \emph{limit} of the sequence, denoted
$\lim_n a_n$.

A sequence $\{a_n\}_n$ converges to $\0$ if and only if $\limsup a_n=\0$
if and only if there exist $ b_n\ge a_n$, $b_n$ decreasing, with $\bigwedge_n b_n=\0$.
A sequence $\{a_n\}_n$ converges to $\1$ if and only if $\liminf a_n=\1$.
(Exercise: If $\lim a_n=\lim b_n = \0$ then $\lim (a_n\cup b_n)=\0$).

If $\{a_n\}_n$ is an antichain then $\lim a_n = \0$.

In a measure algebra, if $\lim a_n = \0$ then $\lim m(a_n)=0$.
If $\sum m(a_n)<\infty$ then $\lim a_n =\0$.

\begin{definition}
A Boolean $\sigma$-algebra $B$ is \emph{weakly distributive} if
whenever $\{W_n\}_n$ is a sequence of countable maximal antichains then each $W_n$
has a finite subset $E_n$
such that $\lim_n \bigcup E_n = \1$.
\end{definition}

Equivalently, if for every $k$, $\{a_{kn}\}_n$ is an increasing sequence with
$\bigvee_n a_{kn}=\1$ then there exists a function $f:N\to N$ such that
$\lim a_{k,f(k)}=\1$.

Or, if for every $k$, $\{a_{kn}\}_n$ is a decreasing sequence with
$\bigwedge_n a_{kn}=\0$ then there exists a function $f:N\to N$ such that
$\lim a_{k,f(k)}=\0$.

\begin{definition}
A Boolean $\sigma$-algebra $B$ is \emph{uniformly weakly
distributive} if there exists a sequence of functions $\{F_n\}_n$
such that for each countable maximal antichain $W$, $F_n(W)$ is a finite subset of
W, and if $\{W_n\}_n$ is a sequence of countable maximal antichains then $\lim_n
\bigcup F_n(W_n) = \1$.
\end{definition}

If $B$ is a measure algebra then $B$  is uniformly weakly
distributive: For every $n$, if $W$ is a countable maximal antichain,
let $F_n(W)=E\subset W$ be such that $m(\bigcup E)\ge 1-1/2^n$.

\begin{definition}
A Boolean $\sigma$-algebra $B$ is \emph{concentrated}
if for every sequence $\{A_n\}$ of finite antichains
with $|A_n|\ge 2^n$ there exist $a_n\in A_n$ such that $\lim a_n = \0$.

$B$ is \emph{uniformly concentrated} if there exists a function $F$
such that for each finite antichain
$A$, $F(A)$ is an element of $A$, and if $A_n$ is a sequence of
finite antichains with $|A_n| \ge 2^n$ then $\lim_n F(A_n) = \0$.
\end{definition}

A measure algebra is uniformly concentrated: if $A$ is a finite
antichain, let $F(A)$ be an element of $A$ of least measure
(then $m(F(A))\leq 1/|A|$.)

\begin{theorem} \label{Thm}
A Boolean $\sigma$-algebra $B$ is a measure algebra if and only if
it is weakly distributive and uniformly concentrated.
\end{theorem}

\section{Fragmentations}

\begin{definition}
A \emph{fragmentation} of $B$ is a sequence of subsets
$C_1\subset C_2\subset...\subset C_n\subset ...$ such that
$\bigcup_n C_n =B^+$ and for every $n$, if $a\in C_n$ and $a\le b$ then $b\in C_n$.

A fragmentation is \emph{$\sigma$-finite cc} if for every $n$,
every antichain $A\subset C_n$ is finite.

A fragmentation is \emph{$\sigma$-bounded cc} if for every $n$
there is a constant $K_n$ such that every antichain $A\subset C_n$
has size $\le K_n$.

A fragmentation is $G_\delta$ if for every $n$, no sequence
in $C_n$ converges to $\0$.

A fragmentation is \emph{tight} if whenever $\{a_n\}_n$ is a sequence such that
$a_n\notin C_n$ for every $n$, then $\lim a_n=\0$.

A fragmentation is \emph{graded} if for every $n$, whenever
$a\cup b\in C_n$ then either $a\in C_{n+1}$ or $b\in C_{n+1}$.
\end{definition}

(The name $G_\delta$ comes from the fact that if $B$ is weakly
distributive and the fragmentation is $G_\delta$ then the set
$\{\0\}$ is a $G_\delta$ set in the convergence topology.)
A $G_\delta$ fragmentation is $\sigma$-finite cc.

In a measure algebra, the fragmentation defined by
$C_n = \{a: m(a)\ge 1/2^n\}$ has all the above properties.

\begin {lemma}
(a) If $B$ is uniformly concentrated then $B$ has a tight
$\sigma$-bounded cc fragmentation.

(b) If $B$ is weakly distributive and has a tight $\sigma$-finite
cc fragmentation then $B$ is uniformly weakly distributive.
\end{lemma}

Consequently, a weakly distributive, uniformly concentrated
$\sigma$-algebra is uniformly weakly distributive.

\begin{proof}
(a) Let $F$ be a function on finite antichains that witnesses
that $B$ is uniformly concentrated. For every $n$ let
$C_n$ be the set of all $a\neq \0$ such that there is no antichain
$A$ with $|A|\ge 2^n$ and $a\le F(A)$.

It is easy to see that $C_n$ is a fragmentation: If $a\notin\bigcup_n C_n$
then for every $n$ there exists an antichain $A_n$ with $|A_n|\ge 2^n$ and $a\le F(A_n)$.
Hence $a\le \lim F(A_n)=\0$, and so $a=\0$. The fragmentation
is tight
because if $a_n\notin C_n$ then $a_n\le F(A_n)$ for some
$A_n$ and so $\lim a_n=\0$.

It is $\sigma$-bounded cc because if $A\subset C_n$ is
an antichain then $F(A)\in C_n$ and so $|A|<2^n$.

(b) Let $B$ be weakly distributive and let $\{C_n\}_n$ be a tight $\sigma$-finite
cc fragmentation; we shall find the functions $F_n$
witnessing that $B$ is uniformly weakly distributive.

Let $n\in N$, and let $W$ be a countable maximal antichain.
We claim that there exists a finite set $E\subset W$ such that
for every finite $E'\subset W-E$, $\bigcup E'\notin C_n$:
otherwise we find an infinite sequence $\{E_k\}_k$ of pairwise disjoint nonempty subsets
of $W$ with $\bigcup E_k \in C_n$, an infinite antichain in $C_n$.
We let $F_n(W)$ be this $E$.

Now let $W_n$, $n\in N$, be countable maximal antichains. We
claim that $\lim \bigcup F_n(W_n)=\1$. Since $B$ is weakly distributive,
there exist finite sets $E_n\subset W_n$ such that $\lim \bigcup E_n =\1$.
For each $n$ let $a_n = \bigcup E_n -\bigcup F_n(W_n)$.
By the claim above, $a_n\notin C_n$. Because $\{C_n\}_n$ is
tight, we have $\lim a_n=\0$,  and because
$\bigcup F_n(W_n)\ge \bigcup E_n \cap (-a_n)$, we have
$\lim \bigcup F_n(W_n) =
\lim \bigcup E_n=\lim -a_n=\1$.

\end{proof}

\begin{lemma}
A Boolean $\sigma-$algebra is uniformly weakly distributive if and only if it
has a tight $G_\delta$ fragmentation.
\end{lemma}

\begin{proof}
First assume that $B$ is uniformly weakly distributive and let $F_n$ be functions witnessing it.
For each $n$ we let $C_n$ be the set of all $a$ such that for some $k\le n$
$a\cap \bigcup F_k(W)\ne \0$
for every countable maximal antichain $W$.

To show that $\{C_n\}_n$ is a fragmentation, we show that $\bigcup_n C_n=B^+$:
if $a\notin C_n$ for all $n$ then for all $k$ there is a $W_k$ such that
$a\cap \bigcup F_k(W_k)=\0$, and because $\lim (a\cap \bigcup F_k(W_k))=a$
we have $a=\0$.

To show that $\{C_n\}_n$ is tight, let $a_n\notin C_n$ for each $n$.
For each $n$ there is a $W_n$ such that $a_n\cap b_n=\0$ where
$b_n=\bigcup F_n(W_n)$. Since $\lim b_n=\1$, we have $\lim -b_n=\0$,
and because $a_n\le -b_n$, it follows that $\lim a_n=\0$.

To show that $\{C_n\}_n$ is $G_\delta$, let $n\in N$ and let
$\lim a_k=\0$; it suffices to find a $k\in N$ such that $a_k \notin C_n$.
We may assume that $\{a_k\}_k$ is strictly decreasing and let
$W$ be the maximal antichain $\{a_{k-1}-a_k\}_k$ where
$a_0=\1$. Let $E=F_1(W)\cup ... \cup F_n(W)$. There exists
a $k$ large enough so that $a_k\cap \bigcup E=\0$. It follows
that $a_k \notin C_n$.

For the converse, let $\{C_n\}_n$ be a tight $G_\delta$
fragmentation. In view of Lemma 2.2 (b) it suffices to show
that $B$ is weakly distributive.
For every $k$, let $\{a_{kn}\}_n$ be a decreasing sequence with
$\bigwedge_n a_{kn}=\0$. We shall find  a function $f:N\to N$ such that
$\lim a_{k,f(k)}=\0$. Given $k\in N$, there is some $f(k)$ such that
$a_{k,f(k)}\notin C_k$, as $\{C_n\}_n$ is $G_\delta$.
Because $\{C_n\}_n$ is tight, $\lim a_{k,f(k)}=\0$ follows.

\end{proof}

A tight $G_\delta$ fragmentation is essentially unique: if $C_n$
and $C'_n$ are such, then for each $n$ there is a $k$ such that
$C_n\subset C'_k$. (If not, there exists an $n$ such that for all $k$
there is some $a_k\in C_n - C'_k$ and so $\lim a_k = \0$).

In the appendix we use a tight $G_\delta$ fragmentation to construct
a continuous submeasure, thus showing that $B$ is uniformly weakly distributive
if and only if $B$ is a Maharam algebra. The construction of a measure
(in the next two chapters) is considerably more difficult.

\begin{lemma}
If $\{C_n\}_n$ is a $G_\delta$ fragmentation of $B$ and if $B$ is concentrated
then $\{C_n\}_n$ is $\sigma-$bounded cc.
\end{lemma}

\begin{proof}
By contradiction, assume that for some $n$, there exist arbitrarily large
finite antichains in $C_n$, and for each $k$, let  $A_k$ be an antichain
in $C_n$ of size $\ge 2^k$. Since $B$ is concentrated, there exist
$a_k\in A_k$ with $\lim a_k=\0$, a contradiction.
\end{proof}

\begin{lemma}
If $\{C_n\}_n$ is a tight $G_\delta$ fragmentation then for every $n$
there exists a $k>n$ such that for every $c\in C_n$, if $c=a\cup b$
then either $a\in C_k$ or $b\in C_k$.

Thus if $B$ has a tight $G_\delta$ fragmentation then $B$ has
one that is also graded (i.e. $a\cup b\in C_n$ implies that either
 $a\in C_{n+1}$ or $b\in C_{n+1}$.)
\end{lemma}

\begin{proof}
Otherwise, for every $k>n$ there exist $c_k=a_k\cup b_k \in C_n$
such that $a_k\notin C_k$ and $b_k\notin C_k$. By tightness,
$\lim a_k=\lim b_k=\0$ and so $\lim c_k=\0$, a contradiction.
\end{proof}

In conclusion, we proved in this chapter that a weakly distributive uniformly
concentrated Boolean $\sigma-$algebra has a graded
$\sigma-$bounded cc fragmentation.
In the next two chapters we construct a measure on $B$ under
the assumption that $B$ is weakly distributive and has a graded
$\sigma-$bounded cc fragmentation.

\section{Kelley's Theorem}

In this Section we introduce Kelley's condition for the existence
of finitely additive measure on a Boolean algebra. But first we show
how the measure problem reduces to finitely additive measures.

\begin {theorem} (Pinsker, Kelley)
A Boolean $\sigma$-algebra $B$ carries a measure if an only if it
is weakly distributive and carries a finitely additive measure.
\end{theorem}

\begin{proof}
Let $m$ be a finitely additive measure on $B$. We let

$$ \mu(b) = \inf \{\lim_n m(b\cap u_n)\}$$
where the infimum is taken over all increasing sequences $\{u_n\}_n$
with $\bigvee_n u_n=\1$.

We show that $\mu$ is a $\sigma$-additive measure, and if $B$ is
weakly distributive then $\mu$ is strictly positive.

First show that $\mu(a\cup b)=\mu(a) + \mu(b)$ if $a\cap b=\0$.
If $s=\{u_n\}_n$, let $\mu_s(x)=\lim_n m(x\cap u_n)$. Clearly,
$\mu_s(a\cup b)=\mu_s(a)+\mu_s(b)$, and so $\mu(a)+\mu(b)\le
\mu_s(a\cup b)$. Hence $\mu(a)+\mu(b)\le \mu(a\cup b)$.

For each $\varepsilon>0$ there is a sequence $s=s_a=\{u_n\}_n$ such that
$\mu_s(a)\le \mu(a)+\varepsilon$, and similarly $s_b=\{v_n\}_n$.
Let $s=\{u_n\cap v_n\}_n$. Then
$\mu(a\cup b)\le\mu_s(a\cup b)=\mu_s(a)+\mu_s(b)\le \mu_{s_a}(a)+\mu_{s_b}(b)\le
\mu(a) + \mu(b)+2\varepsilon$, and the equality follows.

To show the continuity of $\mu$, let $a_n$ be a decreasing sequence
with $\bigwedge_n a_n=\0$; we show that $\lim \mu(a_n)=0$.
Let $\varepsilon>0$.

Let $M=\lim m(a_n)$, and let $K$ be such that $m(a_K)-M<
\varepsilon$. Let $s=\{u_n\}_n$ where $u_n=-a_n$. As for all
$k,n\ge K$, $m(a_k-a_n)<\varepsilon$, we have, for every
$k\ge K$, $\mu_s(a_k)=\lim_n m(a_k\cap u_n)=\lim_n m(a_k-a_n)
\le \varepsilon$, and hence $\mu(a_k)\le\varepsilon$.

Finally, assume that $B $ is weakly distributive, and let $b\in B$ be such that $\mu(b)
=0$; we show that $b=\0$. As $\mu(b)=0$, there is for each $k$
an increasing sequence $\{u_{kn}\}_n$ with $\bigvee_n u_{kn}=\1$
such that $\lim_n m(b\cap u_{kn})<1/k$.

By weak distributivity there is a function $f$ such that $\lim u_{k,f(k)}=\1$.
Hence $\bigvee_n \bigwedge_{k\ge n} u_{k,f(k)}=\1$. For each $k$ let $a_k=
b\cap u_{k,f(k)}$. We have $\bigvee_n \bigwedge_{k\ge n} a_k =b$.
But because $m(a_k)<1/k$, it follows that $\bigwedge_{k\ge n} a_k =\0$,
and so $b=\bigvee_n \bigwedge_{k\ge n} a_k=\0$.
\end{proof}

In order to state and prove Kelley's Theorem, we now work with
Boolean set algebras, and use the term ``finitely additive''
for measures that are not necessarily strictly positive - when we
need the condition $m(a)>0$ if $a\ne\0$ we call $m$ strictly positive.

Let $B$ be a Boolean set algebra, $B\subset P(S)$ for some set $S$.

\begin{definition}
Let $C$ be a subset of $B^+$. For every finite sequence $s=\langle
c_1,...c_n\rangle$ in $C$, let $\kappa_s=k/n$ where $k$ is the
largest size of a subset $J\subset\{1,...,n\}$ such that $\bigcap_{i\in J} c_i$
is nonempty. The \emph{intersection number} of $C$ is the infimum
$\kappa=\inf \kappa_s$ over all finite sequences $s$ in $C$.
\end{definition}

The sequences $s$ do not have to be nonrepeating.

Note that for any $n_0$, the infimum $\inf \kappa_s$ taken over all
sequences $s$ of length $n\ge n_0$ is still $\kappa$: if $s$ is a sequence
of length $n<n_0$, let $t$ be such that $t\cdot n\ge n_0$, and let $s^*$ be
a sequence we get when repeating each term of $s$ $t$-times. Then
$\kappa_{s^*}=\kappa_s$.

To better understand the significance of the intersection number,
assume that $m$ is a finitely additive measure on $B$, and let
$C\subset B^+$. Let $M$ be such that $m(c)\ge M$ for all $c\in C$.
We show that the intersection number $\kappa$ of $C$ is at least $M$.

Let $s=\langle  c_1,...c_n\rangle$ be a sequence in $C$. For each
$i\le n$, let $K_i$ be the characteristic function of $c_i$, i.e.
$K_i (x) =1$ if $x\in c_i$ and $=0$ if $x\notin c_i$. Let
$g=\sum_{i=1}^n K_i$ and consider $I_s = \int g \, dm$,
the area below the graph of $g$. Since $\int K_i=m(c_i)$,
we have $I_s=\sum_i m(c_i)\ge M\cdot n$ and so (because $m(S)=1$)
$||g||=\max_{x\in S} g(x)\ge I_s\ge M\cdot n$.

Thus there exists some $x\in S$ such that $\sum_i K_i(x)\ge M\cdot n$;
in other words, $x$ belongs to at least $M\cdot n$ members of the sequence.
Hence $\kappa_s\ge M\cdot n/n=M$ and consequently $\kappa\ge M$.

\begin{theorem}
(Kelley) Let $C\subset B^+$ have a positive intersection
number $\kappa$. Then there exists a finitely additive measure $m$ on $B$
such that $m(c)\ge \kappa$ for all $c\in C$.
\end{theorem}

\begin{corollary}
If a Boolean algebra $B$ has a fragmentation $\{C_n\}$ such that
each $C_n$ has a positive intersection number, then $B$
carries a strictly positive finitely additive measure.
\end{corollary}

\begin{proof} (of Corollary) For each $n$ let $m_n$ be positive
on $C_n$. If we let $m(a)=\sum_n m_n(a)/2^n$, $m$ is a strictly
positive, finitely additive measure on $B$.
\end{proof}

To prove Kelley's Theorem and construct a finitely additive measure on $B$
we shall consider the vector space of all bounded functions on $S$
(including all characteristic functions $K_a$ for all $a\in B$) and
find a linear functional $F$ such that $F(\1)=1$, $F(K_a)\ge 0$, and
$F(K_c)\ge \kappa$ for all $c\in C$. Then we let $m(a)=F(K_a)$
for all $a\in B$.

To find the linear functional we use the Hahn-Banach Theorem
(for a proof, see Appendix):

\begin{theorem}
Let $p$ be a function such that $p(x)\ge 0$ for all $x$,
$p(x+y)\le p(x)+p(y)$, $p(\alpha\,x)=\alpha\,p(x)$ for all
$\alpha\ge 0$, and $p(\1)\ge 1$.

Then there exists a linear functional $F$ such that $F(\1)=1$
and $F(x) \le p(x)$ for all $x$.
\end{theorem}

In the rest of this Chapter we give a proof of Kelley's Theorem:

\begin{proof}

Let $C\subset B^+\subset P(S)$ and let $\kappa$ be the intersection number of $C$.
Let $V$ be the vector space of all bounded functions on $S$ with
the supremum norm $||f||=\sup \{|f(x)|: x\in S\}$. We shall find
a linear functional $F$ on $V$ such that $0\le F(K_a)\le 1$
for all $a\in B$, $F(\1)=1$ and $F(K_c)\ge \kappa$ for all
$c\in C$.

Consider the convex hull of the set $\{K_c: c\in C\}$:
$$G=\{\sum_{i=1}^{i=m} \alpha_i K_{c_i}: c_i\in C,\,0\le\alpha_i\le 1,\,\sum \alpha_i=1\}$$

\begin{lemma}
For every $g\in G$, $||g||\ge\kappa$.
\end{lemma}

\begin{proof}
First consider rational coefficients $\alpha_i$: for each
$i\le m$, $\alpha_i=l_i/n$ with $\sum_i l_i=n$, and
$g(x)=f(x)/n$ where $f=\sum_i^m l_i\cdot K_{c_i}$.

Consider the sequence $s$ in $C$
of length $n$ where each $c_i$ is repeated $l_i$ times.
By definition of $\kappa$ there are $k$ terms of $s$
with nonempty intersection such that $k/n\ge\kappa$.
Let $x$ be a point in the intersection; it follows that $f(x)\ge k$.
Hence $g(x))\ge k/n\ge\kappa$.

For arbitrary $\alpha_i$ let $\varepsilon>0$. There are rational approximations
$\beta_i$ of the $\alpha_i$ such that
$\kappa\le ||\sum \beta_i K_{c_i}||\le||\sum \alpha_i K_{c_i}||+\varepsilon$.
Hence $||g||\ge\kappa-\varepsilon$, and so $||g||\ge \kappa$.
\end{proof}

Now let
$$Q=\{\alpha(g-\kappa)+f:g\in G,\,\alpha\ge 0,\,f\ge\0\}.$$
The set $Q$ contains all $K_b$, $b\in B$, (because $K_b\ge\0$)
and is convex: if $f$ and $g$ are in $Q$ and $\alpha+\beta=1$
($\alpha,\,\beta>0$) then $\alpha f+\beta g\in Q$. Clearly,
it suffices to verify this for $f$ and $g$ in $G$ and that is easy.

Let $\delta=1-\kappa$, and let $U$ be the open ball
$\{h:||h||<\delta\}$ of radius $\delta$. Using the vector space
convention $A-B=\{a-b:a\in A,\,b\in B\}$, consider the set

$$U-Q=\{h-\alpha(g-\kappa)-f: ||h||<\delta,\,g\in G,\,\alpha\ge 0,\,f\ge\0\}.$$
The set $U-Q$ is convex, and because $\0\in Q$, we have $U\subset U-Q$,
and so for every $v\in V$ there exists a positive number $\alpha$
such that $\alpha v\in U\subset U-Q$. Now define
$$p(v)=\inf\{\gamma>0:\frac{v}{\gamma}\in U-Q\}.$$
If $\gamma=\alpha+\beta$ and $\frac{x}{\alpha},\frac{y}{\beta}\in U-Q$
then $\frac{x+y}{\gamma}=\frac{\alpha}{\gamma}(\frac{x}{\alpha})+
\frac{\beta}{\gamma}(\frac{y}{\beta})\in U-Q$ by convexity,
and so $p(x+y)\le p(x)+p(y)$.

Clearly, $p(\alpha v)=\alpha p(v)$ for all $\alpha\ge 0$.
Finally, let $\gamma\le 1$; we show that $1/\gamma\notin U-Q$, hence
$p(\1)\ge 1$. If $1/\gamma\in U-Q$ then $1/\gamma = h-f$ (and $h=f+1/\gamma$)
where $h\in U$ and $f\in Q$. Since $f\in Q$, there exists, by the Lemma,
some $x\in S$ such that $f(x)\ge 0$, and so $f(x)+1/\gamma\ge 1/\gamma\ge 1$.
But $h(x)<\delta\le 1$.

Now we apply the Hahn-Banach Theorem to this function $p$.
Note that for all $f\in Q$, $-f\in U-Q$ and hence $p(-f)\le 1$.

There exists a linear functional $F$ such that $F(\1)=1$ and
$F(x)\le p(x)$ for all $x\in V$. If $f\in Q$, then $F(-f)\le p(-f)\le 1$
and so $F(f)\ge -1$. As this is true for all $f\in Q$ and $Q$ is closed under
multiples by all $\alpha\ge 0$, it must be the case that $F(f)\ge 0$
for all $f\in Q$. In particular, $F(K_b)\ge 0$ for all $b\in B$.

Also, if $g\in G$, then $g-\kappa\in Q$ and hence $F(g-\kappa)\ge 0$,
i.e. $F(g)\ge \kappa$. Consequently, $F(K_c)\ge\kappa$ for all $c\in C$.

When we let $m(b)=F(K_b)$ for all $b\in B$, $m$ is a finitely additive measure
on $B$ with $m(c)\ge\kappa$ for all $c\in C$.

\end{proof}

\section{The Kalton-Roberts Proof}

We complete the proof by showing that Kelley's condition
applies:

\begin{theorem}
Let $B$ be a Boolean algebra that has a graded $\sigma$-bounded
cc fragmentation $\{C_n\}$. Then for every $n$, $C_n$ has a positive intersection
number.
\end{theorem}

To prove the theorem, we adapt the Kalton-Roberts proof,
an ingenious combinatorial argument that converts finite bounds
for the size of antichains into positive intersection numbers.

\begin{lemma}
Let $M$ and $P$ be finite sets with $|M|=m$ and $|P|=p\le m$,
and let $k$, $3\le k \le p$ be an integer such that $p/k \ge 15\cdot m/p$.
Then there exists an indexed family $\{A_i: i\in M\}$ such that each $A_i$ is a three
point subset of $P$ and such that for every $I\subset M$ with $|I|\le k$,

$$|\bigcup_{i\in I} A_i|>|I|.$$

It follows that for every $I\subset M$ with $|I|\le k$ there exists a one-to-one
choice function $f_I$ on $\{A_i: i\in I\}$.
\end{lemma}

The last statement of the lemma follows by Hall's ``Marriage Theorem'':

\begin{theorem}
A family $\{A_1,...,A_n\}$ of finite sets has a set of distinct representatives if and only if
$|\bigcup_{i\in I} A_i|\ge|I|$ for every $I\subset \{1,...,n\}$.
\end{theorem}

For a proof of Hall's Theorem, see Appendix.

The proof of the lemma uses a counting argument:

\begin{proof}
Consider the families $\{A_i: i\in M\}$ of three point subsets of $P$.
Let us call such a family \emph{bad} if $|\bigcup_{i\in I} A_i|\le |I|$
for some $I\subset M$, $|I|\le k$. If a family is bad then for some $n$, $3\le n\le k$,
there exist sets $I\subset M$ and $J\subset P$, $|I|=|J|=n$ such that $A_i
\subset J$ for every $i\in I$.

There are $\binom p3$ three-point subsets of $P$ and $\binom n3$ three-point
subsets of $J$. Of the $\binom p3 ^n$ families $\{A_i\}_{i\in I}$ with domain $I$, $\binom
n3 ^n$ are such that $\bigcup_{i\in I} A_i\subset J$. The ratio
of such families (for $3\le n \le p$) is
$(\binom n3/\binom p3)^n \le (n^3/p^3)^n=n^{3n}/p^{3n}$ because
$\binom n3/\binom p3\le n^3/p^3$.

Because there are $\binom mn$ subsets $I\subset M$ of size $n$
and $\binom pn$ subsets $J\subset P$ of size $n$, the probability
that a family $\{A_i\}_{i\in M}$ is bad is at most

$$\Pi=\sum_{n=3}^{n=k} \binom mn \binom pn \frac{n^{3n}}{p^{3n}}.$$
We have $\binom mn \cdot \binom pn \cdot n^{3n}/p^{3n} \le
m^n/n! \cdot p^n/n!\cdot n^{3n}/p^{3n}$. Using $e^x>x^n/n!$ we get
$e^n n!>n^n$, hence $1/n!<e^n/n^n$, and so
$$\frac{m^n}{n!}\cdot \frac{p^n}{n!} \cdot \frac{n^{3n}}{p^{3n}} <
\frac{e^{2n}n^n m^n}{p^{2n}}=(e^2 \cdot \frac{n}{p}\cdot \frac{m}{p})^n.$$

For $n\le k$ we have $e^2\cdot n/p\cdot m/p\le e^2\cdot k/p\cdot m/p\le
e^2/{15}<1/2$ because we assumed $p/k\ge 15m/p$ and because
$2e^2<15$. Therefore
$$\Pi <\sum_{n=3}^{n=k}\left(\frac 12\right)^n <1.$$
Consequently, there exists a family $\{A_i: i\in M\}$ that is not bad,
and so $|\bigcup_{i\in I}|>|I|$ for every $I\subset M$ of size $\le k$.

\end{proof}

We shall now apply the Kalton-Roberts method to prove Theorem 4.1:

\begin{proof}
Let $\{C_n\}$ be a graded $\sigma$-bounded cc fragmentation
of a Boolean algebra $B$, and let us fix an integer $n$. We prove that
the intersection number of each $C_n$ is positive, namely $\ge 1/(30K^2)$
where $K=K_{n+2}$ is the maximal size of an antichain in $C_{n+2}$.

We show that for every $m\ge 100K^2$, and every sequence
$\{c_1,...,c_m\}$ in $C_n$ there exists  some $J\subset m$
of size $\ge m/(30K^2)$ such that $\bigcap_{i\in J} c_i$ is nonempty.

Let $M=\{1,...,m\}$ with $m\ge 100K^2$ and let $c_1,...,c_m\in C_n$.
For each $I\subset M$, let

$$b_I = \bigcap\{c_i:i\in I\}\cap\bigcap\{-c_i:i\notin I\}.$$ The
sets $b_I$ are pairwise disjoint (some may be empty) and
$\bigcup\{b_I: I\subset M\}=\1$. Note that for each $i\in M$,
$\bigcup\{b_I: i\in I\}=c_i$. We shall find a sufficiently large
set $J\subset M$ with nonempty $b_J$.

We shall apply Lemma 4.2. First let $k\ge 3$ be the largest $k$ such that
$k/m < 1/(30K^2)$ (there is such because $3/m\le 3/(100K^2).$
We have $k<m$ and $(k+1)/m\ge 1/(30K^2)$. Then let $p$ be the largest $p\ge k$
such that $p/m<1/K$ (there is such because $k/m<1/K$.)

We verify the assumption of the lemma, $p/k\ge15m/p$
(using $p/(p+1)\ge 3/4$):
$$\frac{p}{k}=\frac{p}{p+1}\cdot \frac{p+1}{m}\cdot\frac{m}{k}\ge \frac{3}{4}\cdot \frac{1}{K}\cdot 30K^2\ge 20K$$
and
$$15\frac{m}{p}\le 15\cdot\frac{p+1}{p}\cdot \frac{m}{p+1}\le 15\cdot \frac{4}{3}\cdot K=20K.$$

Now we apply the Lemma: Let $P=\{1,...,p\}$. There exist three point sets $A_i\subset P$,
$i\in M$, and one-to-one functions $f_I$ on all $I\subset M$ of size $\le k$
with $f_I(i)\in A_i$ for all $i\in I$.

We shall prove that there exists a $J\subset M$ of size $\ge k+1$ (and hence $\ge m/(30K^2)$) such that
$b_J$ is nonempty. By contradiction, assume that there is no such $J$. Then
$$\bigcup\{b_I:|I|\le k\}=\1 \text{ and for each } i\in M, \, c_i=\bigcup\{b_I:|I|\le k \text { and } i\in I\}.$$
For each $i\in M$ and $j\in P$ let
$$a_{ij}=\bigcup\{b_I: |I|\le k,\, i\in I \text{ and } f_I(i)=j\}.$$
Note that for each $i\in M$, $c_i =a_{i,j_1}\cup a_{i,j_2}\cup a_{i,j_3}$ where $A_i=\{j_1,j_2,j_3\}$.

Let $j\in P$. We claim that the $a_{ij}$, $i\in M$, are pairwise disjoint:
If $a_{i_1,j}\cap a_{i_2,j}$ is nonempty, then because the $b_I$ are pairwise disjoint
there is some $I$ such that $i_1\in I$ and $i_2\in I$, and because
$f_I(i_1)=j=f_I(i_2)$ and $f_I$ is one-to-one, we have $i_1=i_2$.
Hence the $a_{ij}$, $i\in M$, are pairwise disjoint, and so only
at most $K$ of them belong to $C_{n+2}$.

Consequently, at most $p\cdot K$ of the $a_{ij}$ belong to
$C_{n+2}$ and because $pK<m$, there exists an $i$ such that
$a_{ij}\notin C_{n+2}$ for all (three) $j\in A_i$.

But then $c_i=a_{i,j_1}\cup a_{i,j_2}\cup a_{i,j_3}\notin C_n$ because
the fragmentation is graded. This contradicts the assumption that $c_i\in C_n$.

\end{proof}

\section{Appendix}

\begin{theorem} Let $V$ be a real vector space and let
$p$ be a function on $V$ such that $p(x)\ge 0$ for all $x$,
$p(x+y)\le p(x)+p(y)$, $p(\alpha\,x)=\alpha\,p(x)$ for all
$\alpha\ge 0$.

Let $W$ be a subspace of $V$ and $f$ a linear functional on $W$
such that $f(x)\le p(x)$ on $W$.

Then there exists a linear functional $F$ such that $F(x)=f(x)$
for all $x \in W$, and
$F(x) \le p(x)$ for all $x$.
\end{theorem}

See Walter Rudin: ``Functional Analysis'', Second Edition (1991), pp. 57-58.

\begin{proof}

Using Zorn's Lemma it suffices to extend $f$ one more dimension: let
$u\notin W$ and show that there exists an $F$ extending $f$ on the
subspace $\{x+\alpha u: \alpha \in \mathbf R\}$.

Let $u\notin W$. For every $x\in W$ we have $f(x)\le p(x)\le p(x-u)+p(u)$ and so

$$f(x)-p(x-u)\le p(u).$$

Let $\gamma$ be the supremum of the left hand side of the inequality
taken over all $x\in W$. We let $F(u)=\gamma$, and $F(x+\alpha u)=
f(x)+\alpha \gamma$, for $x\in W$ and $\alpha\in \mathbf R$. $F$ is a linear
functional and it remains to show that $F(x+\alpha u)\le p(x+\alpha u)$.
Note that this follows from these two inequalities:

$$f(x)+\gamma \le p(x+u),$$
$$f(x)-\gamma \le p(x-u).$$

(Then use $x\pm \alpha u = \alpha (x/\alpha \pm u)$ for $\alpha >0$).

The second inequatity is immediate. For the first one, note that for
every $y\in W$, $f(x+y)\le p(x+u)+p(y-u)$, hence $f(x+y)-p(y-u)\le p(x+u)$,
and so
$$ f(x)+\gamma = f(x)+\sup_y(f(y)-p(y-u))=\sup_y(f(x+y)-p(y-u))\le p(x+u).$$

\end{proof}

\section{Appendix}

\begin{theorem}
A family $\{A_1,...,A_n\}$ of finite sets has a set of distinct representatives if and only if
$|\bigcup_{i\in I} A_i|\ge|I|$ for every $I\subset \{1,...,n\}$.
\end{theorem}

See B\'ela Bollob\'as: ``Modern Graph Theory'' (1998), pp. 77-78.

\begin{proof}

The condition is clearly necessary; we prove the sufficiency by induction.
As $n=1$ is obvious, assume that the theorem is true for all $k<n$ and
let $\{A_1,...,A_n\}$ be a sequence of finite sets that satisfy the condition.

Suppose first that for any $k<n$, for any $I\subset \{1,...n\}$ with $|I|=k$,
$|\bigcup_{i\in I} A_i|\ge k+1$. Then we choose $a_n\in A_n$ arbitrarily,
and apply the induction hypothesis to the family $\{A_1-\{a_n\},...,A_{n-1}-\{a_n\}\}$
to get distinct representatives $a_1,...a_{n-1}$ for $A_1,...,A_{n-1}$.

Otherwise, there exists a set $S\subset\{1,...,n\}$ of size $k<n$ such that
$|\bigcup_{i\in S}A_i|= k$. Let $A=\bigcup_{i\in S}A_i$. Consider the family
$\{A_i - A : i\notin S\}$. This family of $n-k$ sets satisfies the condition of the
theorem: if $I\subset \{1,...,n\}-S$ has size $m$ then $|\bigcup_{i\in I} A_i-A|\ge m$
because $|\bigcup_{i\in I\cup S} A_i|\ge k+m$. Thus the family has a set of distinct
representatives, all $\notin A$, and since $\{A_i : i\in S\}$ has distinct representatives
in $A$, we can combine these two sets.

\end{proof}

\section{Appendix}

\begin{definition}
A \emph{a continuous strictly positive submeasure} on a Boolean
$\sigma$-algebra $B$ is a real valued function $m$ on $B$ such
that
\begin{itemize}
\item [(i)] $m(\mathbf 0) = 0,\ m(a) > 0 \text{ for } a \not = \mathbf 0
        \text{, and } m(\mathbf 1) = 1 $
\item [(ii)]$ m(a)\le m(b) \text{ if } a\subset b$
\item [(iii)] $m(a \cup b)\le m(a) + m(b)$
\item [(iv)] if  $\{a_n\}_n$ is a decreasing sequence in $B$ with
$\bigwedge_{n=1}^\infty a_n = \0$ then $\lim_n m(a_n)=\0.$

\end{itemize}

A \emph{Maharam algebra} is a Boolean $\sigma$-algebra that
carries a continuous strictly positive submeasure.

\end{definition}

\begin{theorem}(Balcar-Jech)
$B$ is a Maharam algebra if and only if it is uniformly weakly distributive.
\end{theorem}

\begin{proof}
The proof that a Maharam algebra is uniformly weakly distributive is exacly
the same as for a measure algebra. To show that the condition is sufficient
let $B$ be a $\sigma-$algebra and assume that $B$ has a graded $G_{\delta}$
fragmentation $\{C_n\}$. We shall define a submeasure on $B$.

For each $n$ let $U_n = B-C_n$ and $U_0=B$; we have $U_0\supset U_1\supset
...\supset U_n\supset ...$ and $\bigcap_n U_n = \{\0\}$.
For $X, Y\subset B$, let $X\vee Y$ denote the set $\{x\cup y: x\in X, y\in Y\}$.
As the fragmentation is graded, we have
$$U_{n+1}\vee U_{n+1}\subset U_n$$
for all $n\ge 0$. Cosequently, if $n_1<...<n_k$ then
$$U_{n_1+1}\vee...\vee U_{n_k +1}\subset U_{n_1}.$$
This is proved by induction on $k$.

Let $D$ be the set of all $r=\sum^k_{i=1} 1/2^{n_i}$ where $0<n_1<...<n_k$.
For each $r\in D$ we let $V_r=U_{n_1}\vee...\vee U_{n_k}$ and $V_1=U_0=B$.
For each $a\in B$ define
$$m(a)=\inf\{r\in D\cup\{1\}: a\in V_r\}.$$
Using the above property of the $U_n$, it follows that $V_r\subset V_s$ if $r<s$, and
$V_r\vee V_s\subset V_{r+s}$ when $r+s\le 1.$ From this we have $m(a)\le m(b)$
if $a\subset b$, and $m(a\cup b)\le m(a)+m(b).$

The submeasure $m$ is strictly positive because if $a\ne \0$ then for some $n$,
$a\notin U_n,$ and hence $m(a)\ge 1/2_n.$

The submeasure $m$ is continuous because the fragmentation is $G_\delta:$
If $\{a_n\}_n$ is a decreasing sequence in $B$ converging to $\0$ then for every $k,$
eventually all $a_n$ are in $U_k$, hence $m(a_n)\le 1/2^k$ for eventually all $n,$
and so $\lim_n m(a_n)=0.$

\end{proof}

\bibliographystyle{plain}

\end{document}